\documentclass[12pt]{amsart}
\usepackage{amsmath,amscd,amsthm,amsfonts, amssymb,amsxtra}
\usepackage[all]{xy} \SelectTips{cm}{}
\usepackage{enumitem}
\usepackage{hyperref}
  \hypersetup{colorlinks=true,citecolor=blue}
  \usepackage{graphicx}
\CDat
\newtheorem{thm}{Theorem}[section]  
\newtheorem*{un-no-thm}{Theorem}
\newtheorem{cor}[thm]{Corollary}     
\newtheorem{lem}[thm]{Lemma}         
\newtheorem{prop}[thm]{Proposition}

\newtheorem{bigthm}{Theorem}

\newtheorem{bigadd}[bigthm]{Addendum}

\theoremstyle{definition}
\newtheorem{defn}[thm]{Definition}   

\theoremstyle{definition}

\theoremstyle{definition}
\theoremstyle{remark}
\newtheorem{rem}[thm]{Remark}

\newtheorem*{ack}{Acknowledgement}

\newtheorem*{out}{Outline}

\newtheorem*{intro-rem}{Remark}
\newtheorem*{intro-rems}{Remarks}

\newtheorem{ex}[thm]{Example}

\DeclareMathOperator*{\hocolim}{hocolim}
\DeclareMathOperator*{\colim}{colim}

\begin{document}\
\title{Charged Spaces}
\date{\today} 
\author{John R. Klein} 
\address{Department of Mathematics, Wayne State University,
Detroit, MI 48202} 
\email{klein@math.wayne.edu} 
\author{John W. Peter} 
\address{Department of Mathematics, Utica College,
Utica, NY 13502}
\email{jwpeter@utica.edu}
\begin{abstract} Let $C$ be a model category with an initial object 
$\emptyset$ and functorial factorizations. Let $S\! : C \to C$ be the suspension functor.
An object $X$ of $C$ is said to be {\it charged} if it comes equipped
with a map $S\emptyset \to X$.
If $Y$ is any object of $C$, then $SY$ has a preferred charge, given by 
applying suspension to the map $\emptyset \to Y$. This motivates the question of whether 
a given charged object is a suspension up to a weak equivalence in a way that preserves  charge structures. 
We study this question in the context
of spaces over a given space, where we give a complete obstruction in a certain metastable range. As an application we show how this can be used to study when an embedding 
into a smooth manifold of the form $N\times I$ compresses to an embedding into $N$.
\end{abstract}
\thanks{The first author is partially supported by the 
National Science Foundation}
\maketitle
\setlength{\parindent}{15pt}
\setlength{\parskip}{1pt plus 0pt minus 1pt}
\def\Top{\bold T\bold o \bold p}
\def\wTop{\text{\rm w}\bold T}
\def\wT{\text{\rm w}\bold T}
\def\vo{\varOmega}
\def\vs{\varSigma}
\def\smsh{\wedge}
\def\esmsh{\,\hat\wedge\,}
\def\flush{\flushpar}
\def\dbslash{/\!\! /}
\def\:{\colon\!}
\def\Bbb{\mathbb}
\def\bold{\mathbf}
\def\cal{\mathcal}
\def\orb{\cal O}
\def\hoP{\text{\rm ho}P}

\setcounter{tocdepth}{1}
\tableofcontents
\addcontentsline{file}{sec_unit}{entry}

\section{Introduction \label{sec:intro}}

If $Y$ is a topological space then its {\it un}reduced suspension 
\[
SY = C_-Y \cup_Y C_+Y 
\] has
two preferred basepoints 
\[
-,+ \in SY
\]
given by the vertices of each cone $C_\pm Y$.
We say that a space $X$ is  charged if it comes equipped with 
two base points, which we represent as a map $S^0 \to X$, in which $S^0 = \{-1,+1\}$.
The map $S^0 \to X$ is called the structure map of $X$. A space equipped with a choice of 
charge is called a {\it charged space}.

Let $T$ be the category of spaces and let $T(S^0 \to \ast)$
be the category whose objects are charged spaces, where a morphism is a map 
of underlying spaces that is compatible with the structure maps. A morphism of charged
spaces is said to be a {\it weak equivalence} if it is a weak homotopy equivalence when considered as
a map of spaces. Unreduced suspension
can be regarded as a functor 
\[
S\: T \to T(S^0 \to \ast)\, .
\]
The {\it (unreduced) desuspension problem} for $X$ asks whether there is a space $Y$ and a 
weak equivalence of charged spaces $SY \simeq X$. This is not generally the same as the 
corresponding desuspension problem for based spaces: for example, the zero sphere $S^0$ is
the unreduced suspension of the empty space, but it is not weak equivalent to the reduced suspension of any based space.

We are really interested in a fiberwise version of the desuspension problem.
In order to formulate it, 
fix  a map of spaces $f\: A \to B$ and let
 \[
 T(A@> f >> B)
 \]
 be the category of spaces which factorize $f$. This is the category whose objects
 are triples $(Y,r,s)$ in which $Y$ is a space, and 
 $r\: Y \to B$, $s\: A \to Y$ are maps
 such that $f = r\circ s$. To avoid clutter, it is standard to omit the structure maps $r,s$ from the notation, letting $Y$ refer to $(Y,r,s)$. 
  When $f$ is understood, we
 often denote the category by $T(A\to B)$. Throughout this paper, we will
 always assume that $B$ is a path connected space.

 A morphism $(Y,r,s) \to (Y',r',s')$ is a map of spaces $u\: Y \to Y'$ such that $r = r'\circ u$ and $s' = u\circ s$.  A morphism $Y \to Y'$ is said to be a {\it weak equivalence}
  if it is a weak homotopy equivalence of underlying spaces. More generally, $Y$ and $Y'$
  are {\it weakly equivalent}, written $Y \simeq Y'$, if there is a finite zig-zag of weak equivalences connecting $Y$ to $Y'$. By slight abuse of language, we say in this case
  there is a weak equivalence $Y \simeq Y'$.


 In the special case of the projection map $B\times S^0 \to B$, an object $X \in T(B\times S^0 \to B)$ is called a {\it fiberwise charged space} over $B$.
Consider the unreduced fiberwise suspension functor
 \begin{equation} \label{eqn:sb}
 S_B\: T(\emptyset \to B) \to T(B\times S^0 \to B)
 \end{equation}
which maps an object $Y$ to the object
\[
S_BY = (B \times \{-1\}) \,\, \cup Y \times [-1,1] \cup \,\, (B\times \{1\})\, ,
\]
where the right side is the double mapping cylinder of the structure map $Y \to B$ with itself.
The map $B\times S^0 \to S_B Y$ is given by the two summands appearing in the mapping cylinder.
We are now in the position of being able state the main problem of interest.
\medskip

\noindent {\it Fiberwise Desuspension Problem:} Let $X \in T(B\times S^0 \to B)$
be an object. 
Find an object $Y \in T(\emptyset \to B)$ and a weak equivalence of fiberwise
charged spaces
\[
X\simeq S_B Y \, .
\]


\begin{rem} This problem is about the degree to which the {\it un}reduced fiberwise
suspension functor $S_B\: T(\emptyset \to B) \to T(B\times S^0\to B)$ is 
surjective up to weak equivalence. 
At the risk of belaboring the point, we wish to emphasize that this problem
is not the same as the based version. To see this, we set
\[
R(B) = T(B @>\text{id} >> B)\, .
\]
This is sometimes called the category
of {\it retractive spaces} over $B$.
The based version of the desuspension problem
considers the extent to which the
{\it reduced} fiberwise suspension  functor 
\[
\Sigma_B\: R(B) \to R(B)
\]
is surjective on objects up to weak equivalence (for the definition of 
$\Sigma_B$ see \S\ref{sec:prelim}). A solution $Y$ to the unbased version of
the fiberwise suspension problem
need not admit a section $B\to Y$, whereas of course a solution to the based version 
comes equipped with a section.

Even in the case when $X \in T(B\times S^0 \to B)$ is in the image of the forgetful
functor $R(B) \to T(B\times S^0 \to B)$ (where the structure map is induced 
by the projection $B\times S^0 \to B$), it could well be the case that $X$ can be written in the form $S_B Y$ up to weak equivalence, but $X$ might not be of the form $\Sigma_B Z$ up to
weak equivalence. 
\end{rem}

\begin{ex} Let $X$ be the Klein bottle. This fibers over $S^1$, with fiber $S^1$,
 and it is the unreduced fiberwise suspension of the degree two map $\times 2\: S^1 \to S^1$.
Moreover $X \in T(S^1\times S^0 \to S^1)$ lifts to an object of $R(S^1)$.

However, $X$ is not weakly equivalent to the
reduced fiberwise suspension of an object $Y\in R(S^1)$.  
If it were, then $Y \to S^1$ would be weakly equivalent to a fibration over $S^1$ equipped 
with section and fiber $S^0$. Such a fibration is always fiber homotopically trivial and this would imply
that $X$ is weak homotopy equivalent to a torus. This yields a contradiction, so $X$ is not
a reduced fiberwise suspension up to weak equivalence.
\end{ex}

 In \cite{Klein_haef} the first author proved a Freudenthal suspension theorem for the functor
 \eqref{eqn:sb}. In the current paper we will extend this result to 
 a certain metastable range where there will be an obstruction. To formulate the main
 result, we say that $X\in T(B\times S^0 \to B)$ is {\it $r$-connected}
  if the structure map $X \to B$ is 
 $(r+1)$-connected. We say that $X$ has {\it dimension $\le k$} if $X$ can be obtained up to weak equivalence from 
 $B\times S^0$ by attaching cells of dimension $\le k$. In this case we write
 $\dim X \le k$.
 
The retractive space category $R(B)$ is a pointed simplicial model category (\S\ref{sec:prelim}).
  Hence for objects $U,V \in R(B)$
 we can talk about the abelian group of stable retractive fiberwise homotopy classes
 \[
 \{U,V\}_{R(B)}\, .
 \]
 The trivial element is represented by the class
 of the composite $U \to B \to V$.

To any object $X \in T(B\times S^0\to B)$, we can associate a pair of objects 
\[
i_-X,i_+X \in R(B)
\]
in which $i_-X$ is just $X$ with section $B \to X$ given by 
$B \times -1 \to  B \times S^0 \to B$, and $i_+X$ is similarly defined by restricting to 
$B \times +1 \to B \times S^0$. The category $R(B)$ has internal (fiberwise) smash products, and we therefore  consider the object $i_+X \smsh_B i_-X \in R(B)$. 
We will also need to consider the object of $R(B)$ given by the amalgamated union
\[
X^+ = X \cup_{B\times S^0} B\, .
\]
In \S\ref{sec:coH} we construct a {\it fiberwise reduced diagonal map,}
\begin{equation} \label{eqn:reduced_diagonal} 
\tilde \Delta\: X^+ \to i_+X \smsh_B i_-X\, ,
\end{equation}
which is a morphism of $R(B)$.
Note that one should ``derive'' these constructions
to ensure homotopy invariance. 
To obtain homotopy invariance for 
\eqref{eqn:reduced_diagonal},  
it suffices to assume at the outset that $B\times S^0 \to X$
is a cofibration and $X\to B$ is a fibration. 
In the following discussion, this will be assumed.

It is not difficult to check that the homotopy class of $\tilde \Delta$ is trivial
whenever $X$ is weakly equivalent to  a fiberwise suspension $S_BY$. Our first main result provides a partial converse, which one can view as an unbased fiberwise version of a result of Berstein-Hilton and Ganea \cite{Berstein-Hilton}, \cite{Ganea}.

 \begin{bigthm} \label{bigthm:main}  Let  $X \in T(B\times S^0\to B)$ be an object.
 Assume $X$ is $r$-connected,
 $\dim X \le 3r$ and 
 \[
 [\tilde \Delta] \in \{X^+,i_+X \smsh_B i_-X\}_{R(B)}
 \]
 is trivial. Then 
 $X \simeq S_B Y$. 
 \end{bigthm}
 
  \begin{rem} According to Lemma \ref{lem:pm} below, there
   is a preferred weak equivalence $\Sigma_B i_-X \simeq \Sigma_B i_+X$. Consequently, there is an isomorphism of abelian groups
  \[
  \{X^+,i_+X \smsh_B i_-X\}_{R(B)}  \,\, \cong\,\, \{X^+,i_+X \smsh_B i_+X\}_{R(B)}\, .
 \]
  \end{rem}
  
\subsection*{The relative case} Suppose that $A \to B$ is a map and 
$X \in T(S_B A \to B)$ is an object. The relative case of the fiberwise desuspension problem
asks the extent to which $X$ lies in the image of the functor
\[
S_B \: T(A\to B) \to T(S_BA \to B)
\]
up to weak equivalence (we recover the absolute case when $A = \emptyset$). In
this more general context, we redefine
\[
X^+ \,\, := \,\, X \cup_{S_B A} B \, .
\]
With respect to this notational convention, 
the diagonal obstruction in the relative case
can be regarded as lying in $[\tilde \Delta] \in \{X^+,i_+X \smsh_B i_-X\}_{R(B)}$.
In the current context, we say that $X$ is $r$-connected  if $X\to B$ is $(r+1)$-connected, and we write 
$\dim X \le k$ if $X$ is obtained up to weak equivalence from $S_B A$ by attaching cells of dimension at most
$k$. In the following, we will need to assume that the structure map $S_B A \to X$ is
a cofibration and the structure map $X\to B$ is a fibration.

\begin{bigadd} \label{bigadd:main} With respect to these conventions,
Theorem \ref{bigthm:main} holds in the relative case.
\end{bigadd}

 \subsection*{Embedding up to homotopy type}
 Let $N$ be a compact smooth manifold of dimension $n$. Suppose
 that $f\: K \to N$ is a map, in which $K$ is a finite CW complex of dimension $\le k$. 
 An {\it $h$-embedding} of $f$ 
 consists of a pair $(U,h)$ in which
  $U\subset \text{int} N$ is a compact codimension zero submanifold  and 
   $h\: K @>>> U$  is a homotopy equivalence such that the composite
 \[
 K @> h >> U \subset N
 \]
 is homotopic to $f$. Heuristically, we think of $U$ as a ``compact regular neighborhood'' of $K$ in $N$,  in the sense that $U \subset N$ is a codimension zero compact manifold model for $K$ up to homotopy.

 There is also a notion of concordance of 
 $h$-embeddings of $f$. 
Let  $(U_0,h_0)$ and $(U_1,h_1)$ be $h$-embeddings of $f\: K \to N$. 
A {\it concordance} between them consists of a
compact submanifold $W\subset N \times [0,1]$ and a weak equivalence
 $H\: K\times [0,1] \to W$ such that
 \begin{itemize}
 \item There is a decomposition
 \[
 \partial W = U_0 \cup \partial_1 W \cup U_1\, ;
 \]
 where $U_0,\partial_1 W$ and $U_1$ are compact codimension
 zero submanifolds of $\partial W$ such that $\partial_1 W$ is
 an $h$-cobordism between $\partial U_0$ and $\partial U_1$.
 \item  $W$ meets $N\times \{i\}$ transversly 
with intersection $U_i$ for $i = 0,1$;
 \item $H$ extends $h_0 \amalg h_1$;
 \item the composite $K \times [0,1] @> H >> W \to N\times [0,1]$ is homotopic
 to $f \times \text{id}$.
 \end{itemize}
The operation of concatenation shows that concordance defines an equivalence relation on 
$h$-embeddings of $f$.

The above notion of $h$-embedding often appears in the literature
as ``embedding up to homotopy,'' or as ``embedded thickening.'' 
The topic was first systematically
studied by Wall \cite{Wall_classification4}, Stallings  (in the PL case)
\cite{Stallings} and Mazur
\cite{Mazur}. The work of Wall and Stallings showed that $h$-embeddings exist
whenever $f$ is $(2k-n+1)$-connected and $k \le n-3$.

 Habegger \cite{Habegger} extended this result to one more dimension (i.e.,
$f$ is $(2k-n)$-connected) by exhibiting a necessary and sufficient obstruction
lying in a quotient of a singular cohomology group (cf.\ Remark \ref{rem:Habegger} below) .
 The  case when $N$ is a sphere was studied by Connolly and Williams \cite{Connolly-Williams} where it is shown that the operation which sends an 
$h$-embedding $(U,h)$ of $K$ in $S^n$ to its complement data $S^n \setminus U$
induces, in a wide range, a bijection between the set of concordance classes of $h$-embeddings
of $K$ in $S^n$ and the
set of homotopy types of the Spanier-Whitehead $(n-1)$-duals of $K$. More recently, the second author's Ph.~D. thesis has extended the Connolly-Williams 
program to arbitary compact manifolds $N$ \cite{JP-thesis}.

When $K$ itself is a closed manifold of dimension $\le n-3$ and $n \ge 6$,
the foundational results of surgery theory imply
that $f$ is homotopic to a smooth embedding provided that an $h$-embedding of $f$ exists
and the stable
normal bundle of $f$, i.e., $f^\ast\nu_N - \nu_K$, destabilizes in a suitable way
to a vector bundle of rank $n-k$ (cf.\ \cite[chap.~11]{Wall_book}). In the metastable range $3(k+1) \le 2n$,
there is no obstruction to finding the bundle destabilization, so $f$ is homotopic to an embedding if and only
if an $h$-embedding of $f$ exists.

 Given an $h$-embedding of $f$, one can associate an $h$-embedding
 of the composite 
 \[
 f_1\: K @> f >> N \times 0 \subset N \times D^1
 \]
as follows: we let $J = [-1/2,1/2] \subset D^1$. Then we have 
$U \times J \subset \text{int}(N \times [-1,1])$
and the map 
\[
h_1\: K \times 0 @> h >>  U \times 0 \subset U \times J
\] 
is a homotopy equivalence. Hence,
\[(U \times J,h_1)\] is an $h$-embedding of $f_1$. This is called the 
{\it decompression} of $(U,h)$. Conversely, we say $(U,h)$ is the {\it compression}
of $(U \times J,h_1)$.

 Let $C$ be the closure of the complement of $U$ in $N$ and let
$W$ denote the closure of the complement of $U \times J$ in $N \times D^1$. Then $W \in
T(N \times S^0 \to N)$ is an object and there is a weak equivalence
\[
W \simeq S_B C
\]
(see e.g., \cite{Klein_haef}). Consequently, a necessary condition for an $h$-embedding
of   $f_1\: K \to N \times D^1$
to decompress (up to concordance) is that the complement data must
be weak equivalent to an unreduced fiberwise suspension of some object 
$C\in T(\emptyset \to N)$. 
 
In what follows we set
\[
(K\times K)^+ := (K \times K) \amalg (N \times N) \in R(N\times N)
\]
If $\Delta\: N \to N\times N$ is the diagonal map and $Y \in R(N)$ is an object, we
 set 
  \[
 \Delta_*(Y) := Y \cup_\Delta (N \times N) \in R(N\times N)
 \]
 Lastly, $\Sigma_N^{\tau_N-\epsilon} Y$ denotes the fiberwise suspension of $Y$ with
 respect to the virtual vector bundle $\tau-\epsilon$, where $\tau$ is the tangent bundle of $N$ and $\epsilon$ is the trivial bundle of rank one (cf.\ \S \ref{sec:prelim}).

\begin{bigthm}[Compression] \label{bigthm:compression} 
Let $(U,h)$ be an $h$-embedding of $f_1\: K \to N \times D^1$. Then there is an obstruction 
\[
\theta(U,h) \in \{(K\times K)^+,\Delta_\ast(\Sigma_N^{\tau_N-\epsilon} i_+S_N K)\}_{R(N\times N)}
\]
which vanishes whenever $(U,h)$ compresses into $N$. Conversly, if  we assume in addition  that
\[
r \ge \max(3k-2n+3,\frac{2k-n+3}{2})\, ,
\]
$k\le n-3$ and $n \ge 6$, then
the vanishing of $\theta(U,h)$ is sufficient to compressing $(U,h)$ up to concordance.
\end{bigthm}

\begin{rem} In the special case $N = D^n$, it is not hard to deduce Theorem \ref{bigthm:compression} from the main result of \cite{Klein_sphere}.
\end{rem}

\begin{out} \S\ref{sec:prelim} is mostly about language. In \S\ref{sec:coH}
which introduce the concept of a charged co-$H$ space. This is subsequently used to 
give a proof of Theorem \ref{bigthm:main}. In \S\ref{sec:duality} we recall the notion
of fiberwise duality introduced in \cite{Klein_dualizing2}. \S\ref{sec:compress} contains
the proof of Theorem \ref{bigthm:compression}. 

\end{out}

\begin{ack} The first author wishes to thank the University of Copenhagen
for its hospitality when the research for this paper was being done.
\end{ack}

 \section{Preliminaries \label{sec:prelim}}
 
Henceforth, we let $T$ denote the category of compactly generated weak Hausdorff spaces which we
 consider as enriched over simplicial sets. Products are retopologized with respect to the compactly generated topology.  We assume that the reader is familiar with the basic ideas of model categories. We give $T$ the simplicial  model structure in which a weak equivalence is a weak homotopy equivalence, a fibration is a Serre fibration and a cofibration satisfies the left lifting property with respect to the acyclic fibrations.

As in the introduction,
if $f\: A\to B$ is a map, we let $T(A @>f >> B)$ be the category of factorizations. When $f$ is understood, we abbreviate the notation to $T(A\to B)$.
Define fibrations, cofibrations and weak equivalences in $T(A\to B)$ by applying the forgetful functor $T(A\to B) \to T$. Since 
\[
T(A\to B) \cong f\backslash(T/B)\, ,
\]
it follows from \cite[II.2.8,~prop.~6]{Quillen} that $T(A\to B)$ is 
a simplicial model category with respect to these choices.  For objects $U,V\in T(A\to B)$
we therefore have the homotopy set
\[
[U,V]_{T(A\to B)}
\]
which is given by the homotopy classes of morphisms $U^c \to V^f$ in which $U^c$ is a cofibrant
approximation of $U$ and $V^f$ is a fibrant approximation of $V$.

\subsection*{Finiteness} Let $Y \in T(A\to B)$ be an object.  Given a commutative diagram
 \[
 \xymatrix{
 S^{j-1} \ar[r]^{\alpha}\ar[d] & Y \ar[d] \\
 D^j \ar[r] & B
 }
 \]
 in which the structure map appears on the right, we can form the object
 \[
 Y \cup_{\alpha} D^j \in T(A\to B)\, .
 \] 
 This is called {\it attaching a cell} to $Y$ along $\alpha$.

An object $X \in T(A\to B)$ is said to be {\it finite} if it is, up to isomorphism,
obtained from $A$ by the iterated attachment of a finite number of cells.
More generally, $X$ is {\it homotopy finite}
if it is homotopy equivalent to a finite object. Lastly, $X$ is {\it finitely dominated} if it is a retract of
a homotopy finite object.

\subsection*{Adjunction rules}
For $f\: A\to B$ as above, consider the faithful embedding
\[
i\: R(A) \to T(A \to B)
\]
given by the identity. This has a right adjoint $j\: T(A\to B) \to R(A)$ given by
$Y \mapsto  A\times_B Y$ in which the latter denotes the fiber product.

Let 
\[
f_\sharp\:  R(B) \to T(A \to B)
\]
be defined by $X\mapsto X_\sharp$, in which $X_\sharp$ denotes
$X$ with 
 structure map $A \to X$ given by the composite $A \to B \to X$.
 Then $f_{\sharp}$ has a left adjoint $f^+ \: T(A \to B) \to R(B)$ given by
 $X \mapsto X^+$, where 
  \[
 X^+ = X \cup_{A} B\, .
 \]
There is a third adjunction to consider. Let
 \[
 f^!\: R(B) \to R(A)
 \]
  given by 
 mapping $Y \in R(B)$ to the fiber product $Y^! := Y \times_B A$. This is right adjoint to the functor
 $f_!\: R(A) \to R(B)$ given by 
 $X \mapsto X_!$, where $X_! = X \cup_A B$. 
 Then $f_! = f^+\circ i$ and $f^! = j\circ f_\sharp$.
 The above may be summarized in the diagram 
 \[
 \xymatrix{ 
 R(A) \ar@/^3.0pc/[rrrr]^{f_!}\ar@<1ex>[rr]^i && T(A\to B) \ar[ll]^j \ar@<1ex>[rr]^{f^+} && R(B)\, . \ar[ll]^{f_\sharp} \ar@/^3.0pc/[llll]^{f^!}
 }
 \]
 in which the right pointing arrows are the right adjoints to the corresponding left pointing ones.

 \begin{cor}\label{lem:adjunction_prop} (1). For a cofibrant object $X\in T(A\to B)$ and a fibrant object $Y \in R(B)$ there is 
 a natural isomorphism of pointed sets
 \[
 [X,Y_{\sharp}]_{T(A \to B)} \cong  [X^+,Y]_{R(B)} \, .
 \]
 \noindent (2). For a cofibrant object $Z \in R(A)$, there is an isomorphism of pointed sets
 \[
 [Z,Y^!]_{R(A)} \cong [Z_!,Y]_{R(B)}\, .
 \]
 \end{cor}

 \subsection*{Fiberwise smash products}
If  $X,Y \in R(B)$ are objects, then the {\it internal smash product}
\[
X\smsh_B Y \in R(B)
\]
is given by 
\[
\colim (B @<<< X \cup_B Y @>\subset >> X \times_B Y )
 \]
If $X \in R(B)$ and $Z \in R(B')$, then the {\it external smash product}
\[
X \esmsh Z \in R(B\times B')
 \]
 is given by 
\[
\colim (B\times B' @<<< B \times Z \cup_{B\times B'} X \times B' @>>> X \times Z )
\]
Note that when $B = B'$, we have 
\[
X\smsh_B Y \cong \Delta^!(X\esmsh Y)\, ,
\]
where $\Delta\: B \to B\times B$
is the diagonal map. 
Also note that for $K \in R(B), L\in R(B')$, we have
\[
(K \times L)^+ \cong  K^+ \esmsh L^+
\]
as objects of $R(B\times B')$, where $(K\times L)^+ = (K\times L)\amalg (B\times B')$,
$K^+ = K \amalg B$, and $L^+ = L \amalg B'$.

\subsection*{Fiberwise suspension}
For any map $A\to B$ one can think of
 unreduced fiberwise suspension as a functor
\[
S_B\: T(A\to B) \to T(S_B A \to B)\, .
\] 
 Similarly, the {\it reduced} fiberwise suspension functor $\Sigma_B \: R(B) \to R(B)$ is defined
as
 \[
 \Sigma_B Y := \colim (B @<<<   S_B B @>>>  S_B Y)\, ,
 \]
where on the right we are considering $Y$ as an object of $T(\emptyset \to B)$ by
means of the forgetful functor.
As in the introduction, let $i_-,i_+\: T(B \times S^0\to B) \to R(B)$
be the functors restricting the structure map
using the inclusions $B\times \{\pm 1\} \subset B \times S^0$.

 \begin{lem} \label{lem:pm} Let $X \in T(B\times S^0 \to B)$ be an object.
Then there is a weak equivalence
\[
\Sigma_B i_- X \,\, \simeq \,\, \Sigma_B i_+ X\, .
\]
\end{lem}

\begin{proof} The object $S_B X \in T(B\times S^1 \to B)$ 
can be considered as an object of $R(B)$ by choosing the north pole of $S^1$. With respect
to this choice, the evident
quotient maps $S_B X \to \Sigma_B i_- X$ and $S_B X \to \Sigma_B i_+ X$ are weak equivalences
of $R(B)$.
\end{proof} 
 
 \subsection*{Bundle suspensions}
For a vector bundle $\xi$ with base space $B$, let $S^\xi_B\in R(B)$ be the fiberwise
one-point compactification of $\xi$. For an object $Y \in R(B)$, set
\[
\Sigma^\xi_B Y = S^\xi_B \smsh_B Y \in R(B)\, . 
\]
If $\xi = \epsilon$ is a trivial bundle of rank one, then we recover the reduced fiberwise suspension.

\begin{rem} Let $c\: B \to \ast$ be the constant map to a point. Then
\[
c_!S^\xi_B := B^\xi \in R(\ast)
\]
is the Thom space of $\xi$.
\end{rem}

If $\xi$ and $\eta$ are a pair of vector bundles over $B$, and $X,Y\in R(B)$ are (cofibrant) objects, then we define
\[
[X,\Sigma^{\xi-\eta}_B Y]_{R(B)} := [ \Sigma^\eta_B X,\Sigma^{\xi}_B Y]_{R(B)}\, ,
\]
and similarly,
\[
\{X,\Sigma^{\xi-\eta}_B Y\}_{R(B)} := \{ \Sigma^\eta_B X,\Sigma^{\xi}_B Y\}_{R(B)}\, .
\]

\begin{rem} Alternatively, there is a fiberwise spectrum $S^{\xi-\eta}_B$ over $B$ whose fiber at $b\in B$ is the function spectrum $F(S^\eta_b,S^\xi_b)$. One may then define 
 a fiberwise spectrum $\Sigma^{\xi-\eta}_B Y$ to be $S^{\xi-\eta}_B \smsh_B Y$. This is a model for the fiberwise suspension of $Y$ with respect to the virtual bundle $\xi - \eta$.
 \end{rem}

\subsection*{The Adjoint to $S_B$}  The functor $S_B\: T(\emptyset \to B)
\to T(B\times S^0 \to B)$ has a right adjoint 
\[
O_B\: T(B\times S^0 \to B) \to T(\emptyset\to B)\, .
\]
For $Y \in T(B\times S^0 \to B)$, one defines $O_B Y$ to be the space of paths $\gamma\: D^1 \to Y$ which project
to a constant path $D^1 \to Y \to B$ and which satisfy the condition
$\gamma(\pm 1) \in B\times \{\pm 1\}$. The structure map
$O_B Y \to B$ is given by mapping such a path in $Y$ to 
its associated constant path in $B$.

\subsection*{Equivariant versus fiberwise spaces}
Let $G$ be a topological group whose underlying space is cofibrant. Let $R^G(\ast)$ be the category of based left $G$-spaces. Then $R^G(\ast)$ is a  simplicial 
model category
in which a fibration and weak equivalence are defined by the forgetful functor $R^G(\ast) \to T$ and cofibrations are defined by the left lifting property with respect to the acyclic fibrations.  

Then there is a Quillen equivalence
\begin{equation}
\label{eqn:Quillen_equiv}
\xymatrix{
R^G(\ast)\ar@<1ex>[r]^f & R(BG)\ar[l]^g 
}
\end{equation}
in which $f(X) = B(\ast;G;X)$ is the two-sided bar construction and 
$g(Y) = \text{map}_{R(BG)}(EG,Y)$ is the space of maps $EG\to Y$ which cover the identity
map of $BG$ (here $f$ is the left adjoint to $g$). For details, see \cite[cor.~8.7]{Shulman}. In particular the homotopy
categories of $R^G(\ast)$ and $R(BG)$ are equivalent.

 \begin{rem}  As briefly mentioned in the introduction, when making various functorial constructions it is often crucial to ``derive'' them  by
applying fibrant and/or cofibrant replacements when needed 
 to ensure that the result is homotopy invariant (this  
 occurs especially often in \S\ref{sec:duality}). In order to 
 avoid notational clutter, {\it we assume that
 this has been done wherever necessary, but we will not usually indicate
 it in the notation}. We hope this does not lead to any
 confusion. The reader is forewarned.
\end{rem}

\section{Charged co-H structures \label{sec:coH}}

Let $X \in T(B \times S^0\to B)$ be an object. We can then form the space 
\[
i_+ X \vee_B i_- X \, ,
\]
which is the pushout of the diagram $X @< s_+ << B @> s_- >> X$, where
$s_\pm$ are the restrictions of the structure map $B\times S^0 \to X$ to each summand. 
We wish to consider $i_+ X \vee_B i_- X $ as an object of $T(B \times S^0 \to B)$.
This can be achieved by considering
the commutative diagram
of spaces over $B$
\[
\xymatrix{
B \ar[d]_{s_-} & \emptyset \ar[d]\ar[r] \ar[l] & B \ar[d]^{s_+} \\
X & B \ar[l]^{s_+} \ar[r]_{s_-} & X \, .
}
\]
The pushout of the top line is $B \times S^0$ whereas the pushout of the bottom line
is $i_+ X \vee_B i_- X$. So we have a charge structure $B\times S^0 \to i_+ X \vee_B i_- X$.

Let $i_+X \times_B i_-X$ be the fiber product of $i_+X$ with $i_-X$. There is an evident inclusion 
\[
i_+ X \vee_B i_- X @> \subset >>  i_+X \times_B i_-X\, .
\]
If we give $i_+X \times_B i_-X$ the induced charge structure, the inclusion
becomes a morphism
of $T(B \times S^0 \to B)$.  There is also a diagonal morphism
\[
\Delta\: X \to i_+X \times_B i_-X\, .
\]

\begin{defn} A {\it charged co-$H$ structure} on $X$ is a morphism
\[
p\: X  \to i_+X \vee_B i_-X
\]
of $T(B\times S^0 \to B)$ such that the composition
\[
X @> p >> i_+X \vee_B i_-X @> \subset >> i_+X \times_B i_-X
\]
coincides with $\Delta$ up to homotopy, i.e., in $[X,i_+X \times_B i_-X]_{T(B\times S^0\to B)}$.
\end{defn}

Recall that $i_+ X\smsh_B i_-X$ is the pushout of the diagram
\[
B @<<< i_+X \vee_B i_-X @> \subset >>  i_+X \times_B i_-X
\]
Then the charge structure  $B \times S^0 \to i_+ X\smsh_B i_-X$ factors through $B$.
In other words, we can legitimately consider $i_+ X\smsh_B i_-X$ to be an object of 
$R(B)$ without any loss of information.

Hence,  if $X \in T(B\times S^0 \to B)$ is an object, then the composition
\[
X@> \Delta >> i_+X \times_B i_-X @>>> i_+X \smsh_B i_-X 
\]
factors through $X^+ := X \cup_{B\times S^0} B$, and the resulting map
\[
\tilde \Delta\: X^+ @>>> i_+X \smsh_B i_-X 
\]
is a morphism of $R(B)$.

\begin{lem} \label{lem:diagonal=coH} Assume $X$ is fibrant and cofibrant. 
If $X$ can be equipped with a charged co-H structure, then 
\[
[\tilde \Delta]  \in \{X^+, i_+X \smsh_B i_-X\}_{R(B)}
\]
is trivial. Conversely, if $X$ is $r$-connected, $\dim X \le 3r$ and $\tilde \Delta$ is trivial, then $X$ can be equipped with a charged co-$H$ structure.
\end{lem} 

\begin{proof} (Sketch). The diagram
\begin{equation} \label{eqn:cofibration}
(i_+X \vee_B i_-X) @>\subset >> (i_+X \times_B i_-X) @>>> i_+X \smsh_B i_-X 
\end{equation}
is a cofibration sequence of $T(B\times S^0)$. It follows that if $X$ has a charged co-$H$ structure, then the composite
\[
X \to  i_+X \times_B i_-X @>>> i_+X \smsh_B i_-X
\]
is null-homotopic when considered as a morphism of $T(B\times S^0 \to B)$. Now
apply the functor ${}^+\: T(B\times S^0 \to B) \to R(B)$ to obtain the first part of the lemma.

For the second part, use the Blakers-Massey theorem to show that the diagram 
\eqref{eqn:cofibration}
forms a fiber sequence up through dimension $\le 3r+1$. In particular,
the sequence
\[
[X,i_+X \vee_B i_-X]_{T(B\times S^0\to B)} @>>> [X,i_+X \times_B i_-X ]_{T(B\times S^0\to B)} @>>> [X,i_+X \smsh_B i_-X]_{T(B\times S^0\to B)}
\]
is exact, where the third term is a pointed set. By the adjunction property
Corollary \ref{lem:adjunction_prop}, there is an isomorphism of pointed sets 
\[
[X,i_+X \smsh_B i_-X]_{T(B\times S^0\to B)} \cong [X^+,i_+X \smsh_B i_-X]_{R(B)} 
\]
Furthermore, another application of Blakers-Massey theorem 
shows that stabilization induces an isomorphism
\[
[X^+,i_+X \smsh_B i_-X]_{R(B)} \to \{X,i_+X \smsh_B i_-X\}_{R(B)}\, .\qedhere
\]
\end{proof}

For an object $X\in T(B \times S^0\to B)$ consider 
the canonical morphism
\[
c_X \: S_B O_B X \to X\, .
\]

\begin{prop} \label{prop:reformulation} Assume $X$ is 
fibrant and cofibrant. Then $X$ can be given a charged co-$H$ structure if and only if
$c_X$ admits a section up to  homotopy inside $T(B \times S^0\to B)$.
\end{prop}

\begin{proof} It is enough show that there is an $\infty$-cartesian square
of $T(B\times S^0 \to B)$ of the form
\[
\xymatrix{
S_BO_B X \ar[r] \ar[d]_{c_X} & i_+X \vee_B i_-X  \ar[d]^{\cap} \\
X \ar[r]_<<<<<<{\Delta} &  i_+X \times_B i_-X
}
\]
To do this we convert $\Delta$ into a fibration. Let $W_BX$ be the space
of fiberwise paths in $X$. This is the space whose points are paths
$\lambda\: [-1,1] \to X$ such that the projection to $B$ is constant.
The inclusion $X \to W_BX$ given by the constant paths is a weak equivalence of $T(B\times S^0 \to B)$. Furthermore, the map $X \to i_+X \times_B i_-X$ factors as
\[
X @> \subset >> W_B X @> q >> i_+X \times_B i_-X\, ,
\]
where $q$ is the fibration given by $\lambda \mapsto (\lambda(0),\lambda(1))$.
It is therefore enough to identify
 the pullback of the diagram
\begin{equation} \label{eqn:SO_pullback}
W_BX @>>>  i_+X \times_B i_-X @<<< i_+X \vee_B i_-X
\end{equation}
with $S_BO_B X$ up to weak equivalence.
The pullback of \eqref{eqn:SO_pullback} 
may explicitly computed as an amalgmated union
\[
U \cup V
\]
in which $U$ is the space of paths $\lambda\: [-1,1]\to X$ 
such that $\lambda(-1) \in B \times \{-1\}$ and  
the projection $[-1,1] \to X \to B$ is  constant. Similarly 
$V$ is the space of paths $\lambda\: [-1,1] \to X$ such that
$\lambda(1) \in B\times \{+1\}$ and the projection $[-1,1] \to X \to B$ is  constant.
The intersection $U \cap V$ is the space of paths $\lambda\:[-1,1] \to X$ such that
$\lambda(-1) \in B \times \{-1\}$, $\lambda(1) \in B \times \{+1\}$ and the projection 
$[-1,1] \to X \to B$ is constant. Clearly, $U\cap V = O_B X$.
Furthermore, each of the projections $U \to B$, $V\to B$ 
 is a weak equivalence. In other words, $W_BX$ coincides up to weak equivalence
 with the homotopy pushout of the diagram
 \[
 B @<<< O_B X @>>> B \, .
 \]
that is, with $S_BO_B X$. 
\end{proof}

\begin{thm} \label{thm:desuspend} Let $X\in T(B \times S^0\to B)$ be an object and assume $B$ is connected. 
Assume $X$ is charged co-$H$, is $r$-connected, $\dim X \le 3r$ and $r \ge 1$. Then there is an object $Y \in T(\emptyset \to B)$ and a weak equivalence $S_B Y \simeq X$.
\end{thm}

\begin{proof} In what follows we may assume $X$ is both fibrant and cofibrant. 
We adapt the proof of \cite[thm.~2.1]{Klein_susp}. Since $X$ is charged co-$H$,
we can choose a  section up to homotopy $s\: X\to S_B O_B X$. Applying $O_B$ gives a map
$O_Bs\: O_B X \to O_BS_BO_B X$. We also have a canonical map $u_{O_B X}\: O_B X \to 
 O_BS_BO_B X$. Let $Z$ be the homotopy pullback of the diagram
\begin{equation} \label{eqn:pullback}
 O_B X @> u_{O_B X}  >> O_BS_BO_B X @< O_B s << O_B X\, .
 \end{equation}
Then $Z$ is an object of $T(\emptyset\to B)$, and we have an $\infty$-cartesian
square
\begin{equation} \label{eqn:pullback_square}
\xymatrix{
Z \ar[r]^j \ar[d]_i & O_B X \ar[d]^{ O_B s  } \\
O_B X \ar[r]_<<<<{u_{O_B X} } & O_BS_BO_B X
}
\end{equation}
Each map of this square is $(2r-1)$-connected, by a straightforward Blakers-Massey argument
which which omit. Again by the Blakers-Massey theorem, we find that this square is
$(4r-1)$-cocartesian.  Consider the adjoint $\hat j\: S_B Z \to X$.
\medskip

\noindent {\it Claim:} Assume $r \ge 1$. Then the map 
$\hat j\: S_B Z \to X$ is $3r$-connected.
\medskip

Consider the diagram

\begin{equation} \label{eqn:two_squares}
\xymatrix{
S_BZ \ar[rr]^{S_Bj} \ar[d]_{S_Bi} && S_B O_B X \ar[d]^{S_B O_Bs }\ar[rr]^{c_X} && X \ar[d]^{s}\\
S_B O_BX \ar[rr]_{S_B u_{O_BX} } && S_B O_BS_BO_B X \ar[rr]_{c_{S_B O_BX}} && S_BO_B X\, .
}
\end{equation}

The bottom composite appearing in \eqref{eqn:two_squares} is the identity map.
The left-hand square of \eqref{eqn:two_squares}
 is $(4r)$-cocartesian because it is the suspension of 
the $(4r-1)$-cocartesian square \eqref{eqn:pullback_square}. 

If we can show that the 
right-hand square of \eqref{eqn:two_squares} is $(3r+1)$-cocartesian, then
it will follow that the outer square of \eqref{eqn:two_squares} is also $(3r+1)$-cocartesian.
Since the bottom composite is the identity map, and each vertical map of the square
is $2$-connected it will follow from \cite[lem.~5.6]{Klein_haef} that 
the top composite, which is $\hat j$, is $(3r)$-connected, yielding the claim.

The right-hand square of \eqref{eqn:two_squares} is $(3r+1)$-cocartesian by 
a direct application of the dual Higher Blakers-Massey theorem \cite[th.~2.6]{Goodwillie} to the  $3$-cube
\[
\xymatrix{
 & O_B X\ar[rr]\ar'[d][dd] \ar[dl]
   & & B \ar@{=}[dd] \ar[dl]
\\
B\ar[rr] \ar@{=}[dd]
 & & X \ar[dd]
\\
 & O_BS_BO_BX \ar'[r][rr]\ar[dl]
   & & B\ar[dl]
\\
B \ar[rr]
 & & S_BO_BX 
}
\]
in which the map from the top face to the bottom one is induced by $s$.
We leave the details of this to the reader. This establishes the claim.

To complete the proof of Theorem \ref{thm:desuspend}, we only need
to apply \cite[th.~4.2]{Klein_haef} to the $(3r)$-cocartesian square
\[
\xymatrix{
Z \ar[r] \ar[d] & B\ar[d] \\
B \ar[r] & X 
}
\]
This gives a space $Y$ and a $(3r-2)$-connected map $Y \to Z$ such that the composite
\[
S_B Y \to S_B Z \to X
\]
is a weak equivalence.
\end{proof}
 
\begin{proof}[Proof of Theorem \ref{bigthm:main} and Addendum \ref{bigadd:main}]
We only prove Theorem \ref{bigthm:main}, as the extension to the relative case is 
basically the same.  Let $X\in T(B\times S^0 \to B)$. If
$X \simeq S_B Y$ then $X$ admits a charged co-$H$-structure so $[\tilde \Delta]$ is trivial by the first part of Lemma \ref{lem:diagonal=coH}.

Conversely, assume $[\tilde \Delta] = 0$,
with $X$ $r$-connected and $\dim X \le 3r$.  By the second part of Lemma \ref{lem:diagonal=coH}, $X$ admits a charged co-H structure and by Theorem \ref{thm:desuspend},  $X \simeq S_B Y$.
\end{proof}

\section{Duality \label{sec:duality}}

\subsection*{$N$-duality}
Suppose $N$ is a connected compact  manifold with boundary $\partial N$ (or more generally,
a finite Poincar\'e pair). Then
$N \in T(\partial N \to N)$, and $N^+ \in R(N)$ is just
the double $N \cup_{\partial N} N$ (this is also weakly equivalent
to $i_-(S_N\partial N)$).
  
Suppose we are given finitely dominated 
objects $U,U^\ast\in R(N)$ (which we can asssume to be 
fibrant and cofibrant), and an element
\[
d\in \{\Sigma^j_N N^+,U \smsh_N U^\ast\}_{R(N)}\, .
\]
 
\begin{defn}[cf.~\cite{Klein_dualizing2}]\label{defn:duality_defn} The element $d$ is 
said to be an {\it $N$-duality} if the operation $f\mapsto (f\smsh_{N} \text{id})\circ d$ induces an isomorphism
\begin{equation} \label{eqn:formulation1}
\{U,E\}_{R(N)} \cong \{\Sigma_N^j N^+,E\smsh_N U^\ast\}_{R(N)}  
\end{equation}
 for all objects $E\in R(N)$ which are fibrant and cofibrant.
 The integer $j$ is called the {\it indexing parameter} of $d$.
 \end{defn}
 
 \begin{rem} The object $U\smsh_N U^\ast$ need not be fibrant. 
 Hence, it is too much to hope for $d$ to be represented 
 by a fiberwise stable morphism $ \Sigma^j_N N^+ \to U \smsh_N U^\ast$.
 In general, $d$ may be represented by a fiberwise stable morphism
 \begin{equation}\label{eqn:rep}
 \Sigma^j_N N^+ \to (U \smsh_N U^\ast)^f \, ,
 \end{equation}
 where the target is a fibrant approximation of $U\smsh_N U^\ast$. 
 In this case, we call any representative \eqref{eqn:rep} an 
 {\it $N$-duality map.}
 \end{rem}
 
 If $N$ is connected, the above formulation can be re-expressed in terms of the Quillen
 equivalence \eqref{eqn:Quillen_equiv}: 
 choose a universal principal bundle 
 \[
 p\: \tilde N \to N
 \]
 with structure group $G$ (in particular, this identifies 
 $N$ with $BG$, so $G$ models the loop space $\Omega N$). Then, with respect
 to the Quillen equivalence \eqref{eqn:Quillen_equiv}, $U, U^\ast$ correspond to 
 objects $\tilde U,\tilde U^\ast \in R^G(\ast)$ and $d$ corresponds  to an equivariant stable
 homotopy class 
 \begin{equation}\label{eqn:formulation1a}
 \tilde d\in \{\Sigma^j \tilde N/\partial \tilde N , \tilde U \smsh \tilde U^\ast \}_{R^G(\ast)}\, ,
 \end{equation}
 where $\partial \tilde N$ is the pullback of the universal bundle
 along $\partial N \to N$.
 The $N$-duality condition \eqref{eqn:formulation1}
  can then be re-expressed as follows: for all objects $\tilde E \in R^G(\ast)$ the operation $f\mapsto (f\smsh\text{id}) \circ \tilde d$ yields an isomorphism
 \begin{equation} \label{eqn:formulation2}
 \{\tilde U,\tilde E\}_{R^G(\ast)} \cong  \{\Sigma^j\tilde N/\partial \tilde N,\tilde E \smsh \tilde U^\ast\}_{R^G(\ast)} \, .
 \end{equation}
 There is a technical advantage in the equivariant setting: 
 all objects are fibrant.
 So $\tilde d$ admits a representative equivariant stable map
 \[
 \Sigma^j \tilde N/\partial \tilde N  \to   \tilde U \smsh \tilde U^\ast\, .
 \]
 When there is no confusion we will 
 abuse  notation and denote the representative
 by $\tilde d$.
 
In  some of the proofs appearing below,  it will be more convenient to rephrase \eqref{eqn:formulation2}
 in terms of stable function spaces. For cofibrant objects 
 $A,B \in R^G(\ast)$, we let $F^{\text{st}}(A,B)^G$ be the 
 function space of stable, based, equivariant maps; a point in this space is represented
 by a based equivariant map $\Sigma^k A \to \Sigma^k B$ for some $k \ge 0$. 
 Then the set of path  components of this 
 space is identified with $\{A,B\}_{R^G(\ast)}$, and it's not hard to see that 
 \eqref{eqn:formulation2} is equivalent to  the statement that the operation 
 $f\mapsto (f\smsh\text{id}) \circ \tilde d$ yields a weak homotopy equivalence of function spaces
 \begin{equation} \label{eqn:formulation3}
F^{\text{st}}(\tilde U,\tilde E)^G
 \simeq  
 F^{\text{st}}(\Sigma^j\tilde N/\partial \tilde N,\tilde E \smsh \tilde U^\ast)^G 
 \end{equation}
 for any object $\tilde E \in R^G(\ast)$.

 \subsection*{Equivariant/fiberwise duality} There is another kind of duality which the first author has called {\it equivariant duality} \cite{Klein_immersion}. 
 Suppose we are given  cofibrant objects 
 $\tilde X,\tilde Y \in R^G(\ast)$ and a (stable) map of based spaces
\[
\delta\: S^d \to \tilde X \smsh_G \tilde Y\, ,
\]
where $ \tilde X \smsh_G \tilde Y$ is the orbit space of $G$ acting diagonally on the smash product.
We say that $\delta$ is an {\it equivariant duality map} if for all objects $\tilde E\in R^G(\ast)$,
the operation $f\mapsto (f\smsh_G \text{id})\circ \delta$ induces an isomorphism
of abelian groups
\begin{equation} \label{eqn:form2_iso}
\{\tilde X,\tilde E\}_{{R^{G}(\ast)}} \cong \{S^d,\tilde E\smsh_G \tilde Y\}_{{R(\ast)} }\, .
\end{equation}

In terms of the Quillen equivalence \eqref{eqn:Quillen_equiv}, we can reformulate 
\eqref{eqn:form2_iso}
 in the fiberwise setting. 
and $X, Y \in R(N)$ correspond to $\tilde X,\tilde Y\in R^G(\ast)$, then $\tilde X \smsh_G \tilde Y$
is identified with $(X\smsh_N Y)/N$, i.e., the (homotopy) cofiber of 
the structure map $N \to X\smsh_N Y$. Hence, with respect to the 
identifications $\delta$ can be rewritten as
 \[
 \delta'\: S^d \to (X \smsh_N Y)/N \, ,
 \]
 and the isomorphism \eqref{eqn:form2_iso} becomes
 \begin{equation} 
 \{X,E\}_{{R(N)}} \cong \{S^d,(E\smsh_N Y)/N\}_{{R(\ast)} }
 \end{equation}
 for all objects $E\in R(N)$. In this case, we say that $\delta'$ is a {\it fiberwise duality map}.

\subsection*{Relating the two kinds of duality}
We now explain how  $N$-duality is related to equivariant duality.
For the following description,
 we assume the reader is familiar with the basics of the theory of fiberwise spectra.
Let $\tau$ denote the tangent bundle of $N$ and recall that
$S^{\tau}_N \in R(N)$ is the fiberwise one-point compactification of 
$\tau$. If we apply functions into the sphere spectrum $S^0$ fiberwise, we obtain
a fiberwise spectrum $S^{-\tau}_N$ and we can consider the fiberwise smash product
$S^{-\tau}_N \smsh_N N^+$ which we can conveniently denote as $\Sigma^{-\tau}_N N^+$.
If we collapse the ``zero section'' $N \subset \Sigma^{-\tau}_N N^+$ to a point, we obtain
$N^{-\tau}/(\partial N)^{-\tau}$, where $N^{-\tau}$ is the Thom spectrum of the stable normal bundle of $N$ and $(\partial N)^{-\tau}$ is the Thom spectrum of the pullback of $-\tau$ to $\partial N$. Hence we have a degree one stable map 
$$
\alpha\: S^0 \to (\Sigma^{-\tau}_N N^+)/N
$$
which represents the fundamental class of $N$ (since $H_0((\Sigma^{-\tau}_N N^+)/N)$ is isomorphic to 
$H_d(N,\partial N)$, where $d = \dim N$, and coefficients are twisted by the orientation bundle). 

Suppose now that $d\: \Sigma^j_N N^+ \to (U\smsh_N U^\ast)^f$ is any stable morphism of $R(N)$. Consider the composite
\[
\delta\: S^0 @>\alpha >>  (\Sigma^{-\tau}_N N^+)/N @>(\Sigma^{-\tau}_N d)/N >>  
(U \smsh_N \Sigma^{-\tau}_N U^\ast)^f/N \simeq (U \smsh_N \Sigma^{-\tau}_N U^\ast)/N 
\]

\begin{lem} \label{lem:relation} The map $d$ is an $N$-duality if and only if  
$\delta$ is a fiberwise duality.

\end{lem}
\begin{proof} Consider the diagram of abelian groups
{\small
\[
\xymatrix{
\{X,E\}_{R(N)} \ar[r]^(.4)a & \{\Sigma^j_N N^+ ,E\smsh_N Y\}_{R(N)} \ar[d]^b \\
& \{\Sigma^j_N \Sigma^{-\tau}_N N^+ ,E\smsh_N \Sigma^{-\tau}_N Y\}_{R(N)}
\ar[r]_(.52)c & 
\{S^j,(E\smsh_N \Sigma^{-\tau}_N Y)/N\}_{R(\ast)}\, ,
}
\]
} 

\noindent where the homomorphism $a$ is induced by $f\mapsto (f\smsh_N \text{id})\circ d$, the homomorphism
$b$ is given by fiberwise smashing with $S^{-\tau}_N$
and the homomorphism $c$ is induced by collapsing out $N$ and pre-composing with $\alpha$.
It is easy to see that the map $b$ is an isomorphism. The composite $c\circ b\circ a$
is the homomorphism induced by $\delta$. Consequently, the lemma will follow if we can show 
that $c\circ b$ is an isomorphism. 
By the Quillen equivalence \eqref{eqn:Quillen_equiv}, the homomorphism $c \circ b$ corresponds to the isomorphism 
\[
\{\Sigma^j  \tilde N/\partial \tilde N, \tilde E\smsh\tilde Y\}_{R^G(\ast)}
\cong \{S^j, \tilde E\smsh_G \Sigma^{-\tau} Y\}_{R(\ast)}
\]
induced by the equivariant duality map
\[
S^j \to   \tilde N/\partial \tilde N  \smsh_G S^{-\tau}
\]
that induces Poincar\'e duality for $N$ (see \cite{Klein_immersion}, \cite{Klein_dualizing}). Here, $S^{-\tau}$ is the spectrum with $G$-action
which corresponds to $S^{-\tau}_N$ under the Quillen equivalence \eqref{eqn:Quillen_equiv}. 

\end{proof}

We list  some basic properties of $N$-duality  maps.

\begin{enumerate}[label=(\alph*), leftmargin=2cm]
\item If $d\: \Sigma^j_N N^+ \to (U \smsh_N U^\ast)^f$ is an $N$-duality map, then so is the map
\[
d^t\: \Sigma^j_N N^+@>>> (U\smsh_N U^\ast)^f  @> \text{twist} >> (U^\ast\smsh_N U)^f\, .
\]
Hence, there is no ambiguity in saying that
 $U$ and $U^\ast$ are $N$-dual to each other.
\item Suppose that $d'\: \Sigma^k_N N^+ \to (V\smsh_N V^\ast)^f$ is another $N$-duality map.
Then $d$ and $d'$ induce an {\it umkehr correspondence} 
\begin{equation} \label{eqn:umkehr}
\{U,V\}_{R(N)} \cong \{\Sigma^j_N V^\ast,\Sigma^k_N U^\ast\}_{R(N)}\, .
\end{equation}
\item Given a finitely dominated object $U \in R(N)$, there is an integer $j \gg 0$, a finitely dominated object $U^\ast \in R(N)$ and an $N$-duality map $d\: \Sigma^j_N N^+ \to U \smsh_N U^\ast$.
\item  If $U^\ast$ is $N$-dual to $U$ and  $U'$ is another $N$-dual
 to $U$ (with respect to the same indexing parameter $j$), then $U^\ast$ and $U'$ are stably weak equivalent in $R(N)$. This means that there is 
 a weak equivalence $\Sigma_N^k U^\ast \simeq \Sigma_N^k U'$ for $k$ suitably large. In particular, the stable $N$-dual is unique (in fact, up to contractible choice).
 \item Stably, fiberwise duals  preserve homotopy cofiber sequences. This means that if 
 $U \to V \to W$ is a homotopy cofiber sequence of $R(N)$, then there is a homotopy cofiber sequence of corresponding $N$-duals $W^\ast \to V^\ast \to U^\ast$ (where the indexing parameter $j$ is large and is the same for all three objects).
 \item \label{ha} If $U,U^\ast \in R(N\times I)$ are $(N\times I)$-dual with indexing parameter $j = 0$
 and $p\: N\times I \to N$ is the projection, then $p_\ast U$ and $p_\ast U^\ast$ are $N$-dual with
 indexing parameter $j=1$ (the indexing parameter changes because $p_\ast(N\times I)^+ \cong \Sigma_N N^+$).
\end{enumerate}

The proofs  of most these properties have appeared elsewhere (cf.\ 
\cite{Klein_dualizing2}, \cite{Klein_dualizing}, \cite{Klein_immersion}). 
We now explain the umkehr correspondence \eqref{eqn:umkehr}, 
since crucial use is made of it in this paper. The idea is to apply the duality isomorphisms
simultaneously  to get
\[
\{U,V\}_{R(N)} @> \cong >> \{\Sigma^j N^+,V\smsh_N U^*\}_{R(N)}  @< \cong << \{\Sigma^j V^\ast,\Sigma^k U^\ast\}_{R(N)}\, ,
\]
where the first isomorphism is given by $(f\mapsto f\smsh_N \text{id}_{U^\ast})\circ d$
and the second one by $g\mapsto (\text{id}_V \smsh_N g)\circ d'$.

\begin{rem}
Although similar in spirit, the above notions of duality
should not be confused with the fiberwise duality theory of Becker and Gottlieb \cite{Becker-Gottlieb}---
the latter uses a different concept of  finiteness: an object $U\in R(N)$
is {\it Becker-Gottlieb  finitely dominated} if the homotopy fiber of the map $U \to N$ is a finitely dominated space. Our notion of finiteness is more general: for example, if $N$ 
is closed and of positive dimension, then the wedge
$N \vee S^j \in R(N)$ is finite in our sense but not even finitely dominated in the Becker-Gottlieb sense.
\end{rem}

The main source of examples arises from embedding theory.
Suppose we have a codimension zero compact manifold decomposition of $(N,\partial N)$ as
\[
(U,\partial_0 U) \cup (V_0 ,\partial_0 V)\, ,
\]
in which $\partial U = \partial_0 U \cup_{\partial_{01} U} \partial_1 U$,
$\partial V = \partial_0 V \cup_{\partial_{01} V} \partial_1 V$ and
\[
(U, \partial U) \cap  (V, \partial V) = (\partial_1 U,\partial_{01} U)  = (\partial_1 V,\partial_{01} V) \, .
\]
In particular, when $\partial_0 U = \emptyset$ we can think of this as giving a codimension
zero embedding of $U$ in $N$.  For $i = 0,1$, define $U/\!\!/\partial_i U \in R(N)$ to be the pushout of
$N @<<< \partial_i U \to U$. Then there is an $N$-duality map
\begin{equation} \label{eqn:N-dual-from-embedding}
N^+ = N/\!\!/\partial N @>>> U/\!\!/\partial_0 U \smsh_N U/\!\!/\partial_1 U
\end{equation}
which arises as follows: there is an evident fiberwise diagonal map 
\[
(U,\partial U) \to  (U\times_N U,\partial_0 U \times_N U \cup U \times_N \partial_1 U) 
\]
inducing a fiberwise map
 $U/\!\!/\partial U \to U/\!\!/\partial_0 U \smsh_N U/\!\!/\partial_1 U$. If we precompose this
 with the fiberwise ``collapse'' map 
\[
N/\!\!/\partial N \to  N/\!\!/(V\cup \partial_0 U) \cong U/\!\!/\partial U
\]
we obtain a map of the form \eqref{eqn:N-dual-from-embedding}. The proof
that this map is an $N$-duality essentially follows from Lemma \ref{lem:relation}
in conjunction with \cite{Klein_immersion},\cite{Klein_dualizing}, \cite{Klein_dualizing2};
we omit the details.

Now suppose that $f\: K \to N$ is a fixed map and $(U,h)$ is an $h$-embedding of $f$, where
$U\subset N$ and $h\: K @> \sim >> U$. Let $C$ be the closure of the complement of $U$. Then we obtain a diagram
\begin{equation}\label{eqn:emb_diagram1}
\xymatrix{
& \partial U \ar[r] \ar[d] & C \ar[d] &\partial N \ar[l] \ar@{^{`}->}[ld]\\
K \ar[r]_h^\sim & U \ar[r] & N
}
\end{equation}
in which the displayed square is a pushout. In particular, we obtain a manifold
decomposition
\[
(N,\partial N) = (U,\emptyset) \cup (C,\partial N)
\]
Using $f$, we can identify $K^+ = K\amalg N$ with $U^+$. Furthermore, ``excision'' gives a weak equivalence
$U/\!\!/\partial U \simeq i_+(S_N C)$. Using these identifications together with
the $N$-duality map \eqref{eqn:N-dual-from-embedding}, we infer

\begin{lem} \label{lem:K-plus} $K^+ := K \amalg N$ is $N$-dual to $i_+(S_N C)$ with indexing parameter $j = 0$. 
\end{lem}

 If we reverse the roles of $K$ and $C$, we immediately obtain
 
 \begin{lem} \label{lem:C-plus} $C^+ = C \cup_{\partial N} N$ is $N$-dual to 
 $i_+S_N K$ with indexing parameter $j = 0$.
 \end{lem}

 Suppose next that we are given an $h$-embedding $(U,h)$ of $f_1\: K \to N \times I$. Let $W$ be the closure of the complement of $U$. Then we obtain a diagram similar to \eqref{eqn:emb_diagram1}:
 \begin{equation}\label{eqn:emb_diagram2}
\xymatrix{
& \partial U \ar[r] \ar[d] & W \ar[d] &\partial (N\times I) \ar[l] \ar@{^{`}->}[ld]\\
K \ar[r]_h^\sim & U \ar[r] & N\times I\, .
}
\end{equation}
It now follows from Lemma \ref{lem:C-plus} that 
$W \cup_{\partial (N \times I)} N \times I \in R(N\times I)$
is $(N\times I)$-dual to $i_+S_{N\times I} K$ with indexing parameter $j = 0$. Consequently, 
by property \ref{ha} above, we conclude

\begin{lem} \label{lem:w-plus} The object
\[  W^+ := W \cup_{\partial (N \times I)} N =
 p_\ast (W \cup_{\partial (N \times I)} N \times I)  \]  is $N$-dual to 
$p_\ast i_+S_{N\times I} K \simeq i_+ S_N K$ with indexing parameter $1$. \
\end{lem}

\begin{lem}  \label{lem:smash_behavior} Let $d\in \{N^+ ,U \smsh_N U^\ast\}_{R(N)}$ and 
$d'\in \{N^+, V\smsh_N V^\ast\}_{R(N)}$ be $N$-dualities. Then 
\begin{align*}
d\esmsh d' \in \quad & \{N^+\esmsh N^+ , (U \smsh_N U^\ast) \esmsh (V\smsh_N V^\ast)\}_{R(N\times N)} \\
  \cong \quad &\{(N\times N)^+ ,(U \esmsh V) 
\smsh_{N\times N} V^\ast \esmsh U^\ast\}_{R(N\times N)}
 \end{align*}
 is an $(N\times N)$-duality.
 \end{lem}
 
 \begin{proof} Using  the Quillen 
 equivalence \eqref{eqn:Quillen_equiv}, we saw in 
  \eqref{eqn:formulation1a} that representatives of both
 $d$ and $d'$ correspond to $G$-equivariant maps
 \[
 \tilde d\: \tilde N/\partial \tilde N  \to \tilde U \smsh \tilde U^\ast\, ,
 \qquad \tilde d'\: \tilde N/\partial \tilde N  \to \tilde V \smsh \tilde V^\ast \, ,
 \]
 where $G$ is a suitable topological group model for the loop space $\Omega N$.
It suffices to show that the $(G\times G)$-equivariant map
 \begin{equation} \label{eqn:tilde-d-d'}
\tilde N/\partial \tilde N \smsh \tilde N/\partial \tilde N 
 @>\tilde d \smsh \tilde d' >>  \tilde U \smsh \tilde U^\ast \smsh 
 \tilde V \smsh \tilde V^\ast \cong \tilde U \smsh \tilde V \smsh 
 \tilde V^\ast \smsh \tilde U^\ast
 \end{equation}
 satisfies the equivariant duality condition \eqref{eqn:formulation3}
 with respect to the group $G\times G$. Here we are identifying the domain of 
 \eqref{eqn:tilde-d-d'}
 with $\widetilde{N \times N}/\partial (\widetilde{N\times N})$.
 
 Let $\tilde E \in R^{G\times G}(\ast)$ be a cofibrant object. Then
 \begin{align*}
F^{\text{st}}(\tilde U \smsh \tilde V, \tilde  E)^{G\times G}
 & \cong
F^{\text{st}}(\tilde U , F^{\text{st}}(\tilde V, \tilde  E)^{1\times G})^{G\times 1}\, , \\
  & \simeq 
 F^{\text{st}}(\tilde U , F^{\text{st}}(\tilde N/\partial \tilde N, \tilde  E\smsh \tilde V^\ast)^{1\times G})^{G\times 1} \, , \\
 & \cong
F^{\text{st}}(\tilde N/\partial \tilde N , F^{\text{st}}(\tilde U, \tilde  E\smsh \tilde V^\ast)^{G\times 1})^{1\times G} \, , \\
   & \simeq
 F^{\text{st}}(\tilde N/\partial \tilde N , F^{\text{st}}(\tilde N/\partial \tilde N, \tilde  E\smsh \tilde V^\ast \smsh \tilde U^\ast)^{G\times 1})^{1\times G} \, , \\
    & \cong 
 F^{\text{st}}(\tilde N/\partial \tilde N \smsh  \tilde N/\partial \tilde N, 
    \tilde  E\smsh \tilde V^\ast \smsh \tilde U^\ast)^{G\times G}\, ,
    \end{align*}
where the three isomorphisms listed above are given by the evident adjunctions, and
the second and fourth lines are deduced  by the duality condition \eqref{eqn:formulation3}.
 It is not difficult to check that the composite of these identifications
 is induced by the operation arising from the map \eqref{eqn:tilde-d-d'}.

 \end{proof}

Let $\tau_N$ denote the tangent bundle of $N$.
Let $\Delta\: N \to N\times N$ be the diagonal map.
Recall that $\Delta_\ast\: R(N) \to R(N\times N)$ is the functor
given by $Y \mapsto Y \cup_\Delta (N\times N)$.

\begin{lem}\label{lem:induced_dual} Suppose that $d\in \{N^+ ,U\smsh_N U^\ast\}_{R(N)}$
is an $N$-duality. Then there is preferred $(N\times N)$-duality
\[
\hat d\in \{(N\times N)^+, \Delta_\ast \Sigma^{\tau}_N U \smsh_{N\times N} \Delta_\ast U^\ast\}_{R(N\times N)} \, .
\]

\end{lem}

\begin{proof} Fix a universal principal bundle over $N$ with
structure group $G$. By Lemma \ref{lem:relation} and the Quillen equivalence
\eqref{eqn:Quillen_equiv}, it is enough to find
a  $(G\times G)$-equivariant duality map
\[
S^0 \to (\Sigma^{-\tau} U\smsh_{G} (G\times G)_+) \smsh_{G\times G}  (U^\ast 
\smsh_{G} (G\times G)_+)  
\]
(in the above display, 
$-\tau$ appears instead of $\tau$ since  Lemma \ref{lem:relation} 
involves twisted suspension by $-\tau_{N\times N}$).
Let $E \in R^{G\times G}(\ast)$ be an object. In what follows,
$E$ will also be considered as an object of $R^G(\ast)$ by means of 
restriction along the diagonal $G\to G\times G$. 
Then we have a chain of isomorphisms
\begin{align*}
\{(\Sigma^{-\tau} U)\smsh_{G} (G\times G)_+,E\}_{R^{G\times G}(\ast)}
&\cong \{\Sigma^{-\tau} U,E\}_{R^{G}(\ast)}\, ,\\
& \cong \{S^0,E \smsh_{G} U^\ast\}_{R(\ast)}\, ,\\
& \cong \{S^0,E \smsh_{G\times G} (U^\ast\smsh_G (G\times G)_+)\}_{R(\ast)}\, ,
\end{align*} 
in which the first and last of these is given by extension of scalars and the 
middle one arises from the equivariant duality between $\Sigma^{-\tau} U$ and $U^\ast$. 

Specializing to $E = (\Sigma^{-\tau} U)\smsh_{G} (G\times G)_+$, 
we obtain an isomorphism of abelian groups
{\small
\begin{align*}
\{(\Sigma^{-\tau} U)\smsh_{G} (G\times G)_+,& (\Sigma^{-\tau} U)\smsh_{G} (G\times G)_+\}_{R^{G\times G}(\ast)}  \\
& \cong 
\{S^0,((\Sigma^{-\tau} U)\smsh_{G} (G\times G)_+) \smsh_{G\times G} (U^\ast\smsh_G (G\times G)_+)\}\, .
\end{align*}
}
Since the left side contains a preferred element
given by the identity map of $(\Sigma^{-\tau} U)\smsh_{G} (G\times G)_+$,
it follows that the abelian group on the right possesses a preferred element as well. Represent this preferred
element as a stable map
\[
S^0 \to((\Sigma^{-\tau} U)\smsh_{G} (G\times G)_+) \smsh_{G\times G} 
(U^\ast\smsh_G (G\times G)_+)\, .
\]
Then it is straightforward to check that the latter
 satisfies the equivariant duality condition.
\end{proof}

\begin{cor} \label{prop:C-plus-dual} The object $\Delta_\ast(W^+) :=
 W \cup_{\partial (N\times I)} (N\times N)$ is $(N\times N)$-dual to 
\[
\Delta_\ast(\Sigma_N^\tau i_+S_N K)
\]
 with indexing parameter $j = 1$.
\end{cor}

\begin{proof} This follows from
Lemma \ref{lem:w-plus} and  Lemma \ref{lem:induced_dual}.
\end{proof}

\section{Proof of  Theorem \ref{bigthm:compression} \label{sec:compress}}
\begin{proof}[Proof of Theorem \ref{bigthm:compression}]
Suppose that $(U,h)$ is an $h$-embedding of $f_1\: K \to N \times D^1$,
 where $N$ is a
compact manifold with boundary $\partial N$.
Let $W \in T(N \times S^0 \to N)$ denote the complement. If $(U,h)$ compresses into $N$ then 
there is an object $C\in T(\partial N \to N)$ and a weak equivalence
$W \simeq S_N C$  of  $T(S_N\partial N \to N)$.

The main result of the second author's Ph.~D. thesis says that the converse is in fact
 true provided that 
 \[
 2r \ge 2k-n+3
 \]
  $k\le n-3$ and $n \ge 6$, where $f\: K \to N$ is  $r$-connected and 
$K$ has the homotopy type of a CW complex of dimension $\le k$ (cf.\ \cite{JP-thesis}). 
Hence, to compress $(U,h)$ into 
$N$ we only need to find a fiberwise desuspension of $W$.

Using Poincar\'e duality, a cohomology computation  shows that $W$ is an $(n-k-1)$-connected object having
dimension $\dim W \le n-r$ (cf.\  \cite[p.~615]{Klein_haef}). Applying Theorem \ref{bigthm:main}, we see that $W$ fiberwise
desuspends if and only the diagonal obstruction
\[
[\tilde \Delta] \in \{W^+,i_+W\smsh_N i_-W\}_{R(N)}
\]
vanishes. Here we are assuming in  addition that $n -r \le 3(n-k-1) $, i.e, 
$r \ge 3k-2n+3$.

Using the adjunction property, there is an isomorphism of abelian groups
\begin{equation} \label{eqn:pushforward}
\{W^+,i_+W\smsh_N i_-W\}_{R(N)} \cong \{\Delta_\ast (W^+),i_+W\esmsh i_-W\}_{R(N\times N)} 
\end{equation}

The idea now is to apply the umkehr correspondence 
\eqref{eqn:umkehr} to $[\tilde \Delta]$.
By Corollary \ref{prop:C-plus-dual}, $\Delta_\ast(W^+)$ is $(N\times N)$-dual to 
$\Delta_\ast(\Sigma^{\tau_N}_N i_+S_N K)$ with indexing  parameter $j = 1$. By Lemma \ref{lem:K-plus},
$K^+$ is $N$-dual to $i_+S_NW$  with indexing parameter $j=1$. 
As in the proof of Lemma \ref{lem:pm}, $i_+S_NW$ can be identified with $\Sigma_N i_+ W$.
Consequently, $K^+$ is $N$-dual to $i_+W$ with indexing parameter $j=0$. 
By Lemma \ref{lem:pm}, $\Sigma_N i_- W \simeq \Sigma_N i_+ W $, so $K^+$ is also $N$-dual to $i_- W$. By Lemma \ref{lem:smash_behavior}),
we see that $i_- W \esmsh i_+ W$ is $(N\times N)$-dual to $K^+ \esmsh K^+ \cong (K \times K)^+$ with indexing parameter $j=0$. Applying the umkehr correspondence and the isomorphism \eqref{eqn:pushforward} we obtain an isomorphism of abelian groups
\begin{equation} \label{eqn:last-iso}
\{W^+,i_+W\smsh_N i_-W\}_{R(N)} \cong \{(K \times K)^+,\Delta_\ast(\Sigma^{\tau_N-\epsilon}_Ni_+S_N K)\}_{R(N\times N)} \, .
\end{equation}
We can therefore take 
\[
\theta(U,h) \in \{(K \times K)^+,\Delta_\ast(\Sigma_N^{\tau_N-\epsilon} i_+S_N K)\}_{R(N\times N)} 
\]
to be the unique element  that corresponds to $[\tilde \Delta]$ with respect to the isomorphism \eqref{eqn:last-iso}.
\end{proof}
\section{Identification of the critical group}  

Suppose that $f\: K \to N$ is $(2k-n)$-connected. We will identify the obstruction group
\[
\{(K\times K)^+,\Delta_\ast(\Sigma_N^{\tau_N-\epsilon} i_+S_N K)\}_{R(N\times N)} 
\]
in terms of singular cohomology.

The following result follows from classical obstruction theory. We omit the proof.

\begin{lem} \label{lem:obstruction_theory}Let $X,Y \in R(B)$ be objects. Suppose $Y$ is an $j$-connected object and 
$\dim X \le j+1$. Let $r_Y \: Y \to B$ denote the structure map. Then
\[
\{X,Y\}_{R(B)} \cong H^{j+1}(X,B;\pi_{j+1}(r_Y))\, ,
\]
where  the cofficients are twisted with respect to the $\pi_1(B)$-module 
$\pi_{j+1}(r_Y)$.
\end{lem}

Recall that $p\: \tilde N \to N$ is a choice of universal
principal bundle with structure group $G$. If we set  $\Gamma = \pi_0(G)$,
then $\Gamma$ is identified with the fundamental group of $N$. Let
$w\: \Gamma \to \Bbb Z/2$ be the homomorphism given by the first
Stiefel-Whitney class of the tangent bundle $\tau$ of $N$ (this
uses the identification $\hom(\Gamma,\Bbb Z_2) \cong H^1(N;\Bbb Z_2)$).
Define a $(\Gamma\times \Gamma)$-module structure on 
$\pi_{2k-n+1}(f) \otimes \Bbb Z[\Gamma]$ as follows: if $(g_1,g_2) \in G\times G$, $x \in \pi_{2k-n+1}(f)$ and $y \in \Bbb Z[\Gamma]$,
 then 
 \[
 (g_1,g_2) \cdot (x,y) := w(g_1)g_1x\otimes g_2yg_1^{-1}
 \] (compare
 \cite{Habegger}).

\begin{cor} \label{cor:fringe_dim} Assume $f\: K \to N$ is $(2k-n)$-connected.
Then there is an isomorphism of abelian groups
\[
\{(K\times K)^+,\Delta_\ast(\Sigma_N^{\tau_N-\epsilon} i_+S_N K)\}_{R(N\times N)}  
\cong
H^{2k}(K\times K;\pi_{2k-n+1}(f) \otimes \Bbb Z[\Gamma] )\, ,
\]
where coefficients are twisted by the $(\pi_1(K)\times \pi_1(K))$-module structure on $\pi_{2k-n+1}(f) \otimes \Bbb Z[G]$ that is induced by the pullback along the homomorphism 
$f_\ast\times f_\ast\: \pi_1(K)\times \pi_1(K) \to 
\Gamma\times \Gamma$.
\end{cor}

\begin{proof} 
Set $Y = \Sigma_N^{\tau_N-\epsilon} i_+S_N K$. Then 
\[
\Delta_\ast Y = Y \cup_\Delta (N \times N) \simeq \hocolim (Y @<<< N @> \Delta >> N \times N) \, .
\]
As  homotopy pullbacks over the
same base commute with (homotopy) colimits, the pullback to
$\Delta_\ast Y$ of the universal bundle $p\times p\: \tilde N\times \tilde N \to N \times N$ may be identified up to homotopy with the based $(G\times G)$-space
\begin{equation} \label{eqn:identification}
\colim (\tilde Y \times G^{\text{ad}} @<<< \ast \times G^{\text{ad}} @>>> \ast)
\simeq \tilde Y \smsh G^{\text{ad}}_+
\end{equation}
where 
\begin{itemize}
\item $\tilde Y := Y\times_N \tilde N$ is the fiber product of $Y$ and
 $\tilde N$ 
over $N$. This comes equipped with the structure of a based $G$-space. We make it into a based
$(G\times G)$-space by letting the left factor
of $G\times G$ act with the given $G$-action and declaring that the
right factor act trivially.
\item The space $G^{\text{ad}}$ is a copy of
 $G$ with  $(G\times G)$-action  $(g,h)\cdot x := gxh^{-1}$.
\item We give $\tilde Y \smsh G^{\text{ad}}_+$ the diagonal $(G\times G)$-action.
\end{itemize}
Note that the underlying space of  $\tilde Y \smsh G^{\text{ad}}_+$
is identified  with the homotopy fiber of the structure map
$r_{\Delta_\ast Y} \: \Delta_\ast Y \to N\times N$.

In particular, since $\tilde Y$ is $(2k-2)$-connected, we see that
$\tilde Y \smsh G^{\text{ad}}_+$ is
$(2k-2)$-connected, i.e., $r_{\Delta_\ast Y}$ is a
$(2k-1)$-connected map. 
By Lemma \ref{lem:obstruction_theory}, it follows that
\[
\{(K\times K)^+,\Delta_\ast Y\}_{R(N\times N)}  
\cong
H^{2k}(K\times K;\pi_{2k}(r_{\Delta_\ast Y}))\, .
\]
It suffices to show that $\pi_{2k}(r_{\Delta_\ast Y})$ is isomorphic
to $\pi_{2k-n+1}(f) \otimes \Bbb Z[\Gamma]$ as $(\Gamma\times \Gamma)$-modules. 
But this is straightforward to check using the identification
\eqref{eqn:identification}.
\end{proof}

\begin{rem}\label{rem:Habegger} If $f\: K \to N$ is $(2k-n)$-connected, $2k-n \ge 2$ and
$k \le n-3$, then Habegger exhibits
a necessary and sufficient obstruction to finding an $h$-embedding
of $f$ in $N$ \cite{Habegger}. The obstruction lies in the group
\[
H^{2k}(K\times K;\pi_{2k-n+1}(f) \otimes \Bbb Z[\Gamma] )_{\Bbb Z_2}
\]
which is the coinvariants of a certain involution of the obstruction group appearing in Corollary \ref{cor:fringe_dim}. In comparing his 
result with our Theorem \ref{bigthm:compression}, note  that Habegger's 
assumptions $2k-n \ge 2, k \le n-3$ are more restrictive than our 
$k \le n-3, n \ge 6$.
\end{rem}

\end{document}